\documentclass[12pt,english,BCOR7.5mm]{amsart}
\usepackage{ae,aecompl}
\usepackage[T1]{fontenc}
\usepackage[latin9]{inputenc}
\usepackage[letterpaper]{geometry}
\usepackage{color}
\usepackage{babel}
\usepackage{amsthm}
\usepackage{amssymb}
\usepackage[unicode=true,
 bookmarks=true,bookmarksnumbered=true,bookmarksopen=true,bookmarksopenlevel=2,
 breaklinks=false,pdfborder={0 0 1},backref=false,colorlinks=true]
 {hyperref}
\hypersetup{pdftitle={Using XY-pc in LyX},
 pdfauthor={H. Peter Gumm},
 pdfsubject={LyX's XY-pic manual},
 pdfkeywords={LyX, documentation},
 urlcolor=blue, filecolor=blue,pdfpagelayout=OneColumn, pdfnewwindow=true,pdfstartview=XYZ, plainpages=false, pdfpagelabels}

\makeatletter
\numberwithin{equation}{section}
\numberwithin{figure}{section}
\theoremstyle{plain}
\newtheorem{thm}{\protect\theoremname}
  \theoremstyle{plain}
  \newtheorem{conjecture}[thm]{\protect\conjecturename}
  \theoremstyle{definition}
  \newtheorem{defn}[thm]{\protect\definitionname}
  \theoremstyle{plain}
  \newtheorem{cor}[thm]{\protect\corollaryname}
  \theoremstyle{plain}
  \newtheorem{lem}[thm]{\protect\lemmaname}
  \theoremstyle{plain}
  \newtheorem{prop}[thm]{\protect\propositionname}

\usepackage[all]{xy}

\newcommand{\xyR}[1]{
  \xydef@\xymatrixrowsep@{#1}}
\newcommand{\xyC}[1]{
  \xydef@\xymatrixcolsep@{#1}}

\newdir{|>}{!/4.5pt/@{|}*:(1,-.2)@^{>}*:(1,+.2)@_{>}}

\let\myTOC\tableofcontents
\renewcommand\tableofcontents{%
  \pdfbookmark[1]{\contentsname}{}
  \myTOC }

\def\LyX{\texorpdfstring{%
  L\kern-.1667em\lower.25em\hbox{Y}\kern-.125emX\@}
  {LyX}}

\makeatother

  \providecommand{\conjecturename}{Conjecture}
  \providecommand{\corollaryname}{Corollary}
  \providecommand{\definitionname}{Definition}
  \providecommand{\lemmaname}{Lemma}
  \providecommand{\propositionname}{Proposition}
\providecommand{\theoremname}{Theorem}

\begin{document}

\title{Cross Number Invariants of Finite Abelian Groups}

\author{Xiaoyu He}

\email{xiaoyuhe@college.harvard.edu}

\address{Eliot House, Harvard College, Cambridge, MA 02138.}

\date{\today}
\begin{abstract}
The cross number of a sequence over a finite abelian group $G$ is
the sum of the inverse orders of the terms of that sequence. We study
two group invariants, the maximal cross number of a zero-sum free
sequence over $G$, called $\mathsf{k}(G)$, introduced by Krause,
and the maximal cross number of a unique factorization sequence over
$G$, called $K_{1}(G)$, introduced by Gao and Wang. Conjectured
formulae for $\mathsf{k}(G)$ and $\mathsf{K}_{1}(G)$ are known,
but only some special cases are proved for either. We show structural
results about maximal cross number sequences that allow us to prove
an inductive theorem giving conditions under which the conjectured
values of $\mathsf{k}$ and $\mathsf{K}_{1}$ must be correct for
$G\oplus C_{p^{\alpha}}$ if they are correct for a group $G$. As
a corollary of this result we prove the conjectured values of $\mathsf{k}(G)$
and $\mathsf{K}_{1}(G)$ for cyclic groups $C_{n}$, given that the
prime factors of $n$ are far apart. Our methods also prove the $\mathsf{K}_{1}(G)$
conjecture for rank two groups of the form $C_{n}\oplus C_{q}$, where
$q$ is the largest or second largest prime dividing $n$, and the
prime factors of $n$ are far apart, and the $\mathsf{k}(G)$ conjecture
for groups of the form $C_{n}\oplus H_{q}$, where the prime factors
of $n$ are far apart, $q$ is the largest prime factor of $n$, and
$H_{q}$ is an arbitrary finite abelian $q$-group. Finally, we pose
a conjecture about the structure of maximal-length unique factorization
sequences over elementary $p$-groups, which is a major roadblock
to extending the $\mathsf{K}_{1}$ conjecture to groups of higher
rank, and formulate a general question about the structure of maximal
zero-sum free and unique factorization sequences with respect to arbitrary
weighting functions.
\end{abstract}
\maketitle

\section{Introduction}

Let $(G,+)$ be a finite abelian group written additively, and let
$G^{\bullet}$ be the set of nonzero elements of $G$. For any subset
$G_{0}\subset G$, we define $\mathcal{G}(G_{0})$ to be the multiplicative
free abelian group generated by $G_{0}$. Similarly, we define $\mathcal{F}(G_{0})\subset\mathcal{G}(G_{0})$
to be the multiplicative free abelian monoid over $G_{0}$. A \emph{sequence
}over $G_{0}$ is an element of $\mathcal{F}(G_{0})$. Elements of
$\mathcal{G}(G_{0})$ are of the form 
\[
S=\prod_{g\in G_{0}}g^{v_{g}(S)},
\]
where $v_{g}:\mathcal{G}(G_{0})\rightarrow\mathbb{Z}$ is the \emph{valuation}
function for $g$, satisfying $v_{g}(S)=0$ for all but finitely many
$g$ given any fixed $S$. If $S\in\mathcal{F}(G_{0})$ then we have
further that $v_{g}(S)\geq0$ for all $g\in G_{0}$. The identity
$1$ of the monoid $\mathcal{F}(G_{0})$ is the unique sequence satisfying
$v_{g}(1)=0$ for all $g\in G_{0}$. Given two sequences $S,T\in\mathcal{F}(G_{0}),$
we say that $T$ is a \emph{subsequence} of $S$, or \emph{divides}
$S$, if $v_{g}(T)\leq v_{g}(S)$ for all $g\in G_{0}$. In such a
case we may also write $T\mid S$.

By the \emph{greatest common divisor} of two sequences $S$ and $T$
over $G_{0}$ we mean the sequence 
\[
\mbox{gcd}(S,T)=\prod_{g\in G_{0}}g^{\min(v_{g}(S),v_{g}(T))}.
\]

Let $\mathbb{N}=\mathbb{Z}_{\geq0}$. An \emph{indexed sequence} over
a set $G_{0}\subset G$ is a sequence over $G_{0}\times\mathbb{N}$,
i.e. an element of $\mathcal{F}(G_{0}\times\mathbb{N})$. To each
indexed sequence $S\in\mathcal{F}(G_{0}\times\mathbb{N})$ we associate
a unique sequence $\alpha(S)$, where the map $\alpha$ is given by
extending $\alpha((g,n))=g$ multiplicatively to a monoid homomorphism
$\alpha:\mathcal{F}(G_{0}\times\mathbb{N})\rightarrow\mathcal{F}(G_{0})$.
This map $\alpha$ is called the \emph{unlabelling homomorphism. }The
valuation function on $G$ is extended to indexed sequences naturally
by composing it with the unlabelling homomorphism, i.e. $\overline{v}_{g}(S)=v_{g}(\alpha(S))$.
We consider elements of $G_{0}\times\mathbb{N}$ as indexed elements
of $G_{0}$, and we define the \emph{order} of such an element $(g,n)$
to be the order $\mbox{ord}(g)$ of $g$ in $G$, i.e. the smallest
positive integer $m$ such that $mg=0$ in $G$.

A sequence (resp. indexed sequence) over a group $G$ is defined as
a sequence (resp. indexed sequence) over its set of nonzero elements
$G^{\bullet}$.

Define an indexed sequence over $G_{0}$ to be \emph{squarefree} if
$v_{(g,n)}(S)\leq1$ for all pairs $(g,n)\in G_{0}\times\mathbb{N}$.
Unless otherwise specified all the indexed sequences we study are
assumed to be squarefree. 

If $S$ is an indexed sequence over $G$ and $G_{0}$ is a subset
of $G$, then define $S_{G_{0}}$ to be

\[
S_{G_{0}}=\prod_{(g,n)\in G_{0}\times\mathbb{N}}(g,n)^{v_{(g,n)}(S)}.
\]

If $G_{0}$ is a subgroup of $G$ then $S_{G_{0}}$ will simultaneously
be considered an indexed sequence over $G_{0}$.

The \emph{sum }function $\sigma:\mathcal{F}(G^{\bullet})\rightarrow G$
is defined on a sequence $S$ as 
\[
\sigma(S)=\sum_{g\in G^{\bullet}}v_{g}(S)\cdot g.
\]

This function extends naturally to a sum function $\overline{\sigma}:\mathcal{F}(G^{\bullet}\times\mathbb{N})\rightarrow G$
on indexed sequences given by $\overline{\sigma}(T)=\sigma(\alpha(T))$.

Define the \emph{set of subsums}, or \emph{sumset}, of an indexed
sequence $S$ over $G$ to be the set 
\[
\Sigma(S)=\{\overline{\sigma}(T):T\mid S\}.
\]

We will say that $S$ has \emph{full sumset} in $G$ if $\Sigma(S)=G$. 

Overviews of progress on computing sumset sizes of several types of
sequences are made in Chapter 5.2 of \cite{GH}, Section 2 of \cite{GR},
and Part 1 of \cite{Gr}. Specialized versions of this problem, such
as counting the number of subsequences of a given sum \cite{GG1},
and counting sums of only short subsequences \cite{Ge1}, have been
fruitful objects of study.

An indexed sequence $S$ is \emph{zero-sum} if $\overline{\sigma}(S)=0$,
\emph{zero-sum free} if there does not exist $1\neq T\mid S$ with
$\overline{\sigma}(T)=0$, and \emph{irreducible }or \emph{minimal
zero-sum} if it differs from $1$, is zero-sum, and has no nontrivial
proper zero-sum subsequence.

Define the\emph{ monoid of zero-sum indexed sequences over $G$ }to
be the monoid
\[
\mathcal{T}(G)=\{S\in\mathcal{F}(G^{\bullet}\times\mathbb{N}):\overline{\sigma}(S)=0\}.
\]

The set $\mathcal{A}(G)\subset\mathcal{T}(G)$ is defined as
\[
\mathcal{A}(G)=\{S\in\mathcal{T}(G):S\mbox{ is irreducible}\}.
\]

We call $\mathnormal{Z}(G)=\mathcal{F}(\mathcal{A}(G))$ the \emph{factorization
monoid} of $G$. For any $S\in\mathcal{A}(G)$ we denote by $[S]$
the corresponding generator in $\mathnormal{Z}(G)$. Let $\pi:\mathnormal{Z}(G)\rightarrow\mathcal{T}(G)$
denote the monoid homomorphism taking a formal product $\prod_{i\leq m}[S_{i}]$
of irreducibles to the zero-sum indexed sequence $\prod_{i\leq m}S_{i}$.
An \emph{irreducible factorization} of a zero-sum indexed sequence
$S\in\mathcal{T}(G)$ is an element of $\mathnormal{\pi^{-1}(S)}$.
Equivalently, an irreducible factorization of $S$ is a way of writing
$S=\prod_{i\leq m}S_{i}$, where the $S_{i}$ are irreducible subsequences
of $S$. We say that $S$ is a \emph{unique factorization indexed
sequence }(UFIS) if $|\pi^{-1}(S)|=1$.

Zero-sum indexed sequences and UFIS's have interpretations in algebraic
number theory when $G$ is taken to be the ideal class group of the
integer ring of an algebraic number field \cite{BC}.

When $S$ is an indexed sequence over $G$, let $|S|=\sum_{g\in G^{\bullet}}\overline{v}_{g}(S)$.
This is referred to as the \emph{length} of $S$. Define the \emph{cross
number} of an indexed sequence $S$ to be 
\[
\mathsf{k}(S)=\sum_{g\in G^{\bullet}}\frac{\overline{v}_{g}(S)}{\mbox{ord}(g)}.
\]

In algebraic number theory it is often more natural to study the cross
number of an indexed sequence than to study its length \cite{BC}. 

Three cross number invariants of a finite abelian group $G$ are defined
as follows. The \emph{little cross number of $G$} is defined as
\[
\mathsf{k}(G)=\max\{\mathsf{k}(S):S\mbox{ is zero-sum free over \ensuremath{G}}\},
\]
the \emph{cross number of $G$} is defined as
\[
\mathsf{K}(G)=\max\{\mathsf{k}(S):S\mbox{ is irreducible over \ensuremath{G}}\},
\]
and the $\mathsf{K}_{1}$ \emph{constant of $G$ }is defined as
\[
\mathsf{K}_{1}(G)=\max\{\mathsf{k}(S):S\mbox{ is a UFIS over \ensuremath{G}}\}.
\]

The constants $\mathsf{k}(G)$ and $\mathsf{K}(G)$ have been studied
intensely \cite{BC,BCMP,GG2,Ge2,GS1,KZ}. The cross number is itself
a number-theoretically natural alternative to the \emph{Davenport
constant }$\mathsf{D}(G)$ of a group, which is the maximal length
of any irreducible indexed sequence over $G$ \cite{GS2}. Chapter
2 of the recent book \cite{GR} by Geroldinger and Ruzsa summarizes
the main results on the Davenport constant.

Define $P^{-}(n)$ and $P^{+}(n)$ to be the smallest and largest
primes, respectively, dividing $n$. It is not difficult to show that
\[
\mathsf{k}(G)+\frac{1}{\mbox{exp}(G)}\leq\mathsf{K}(G)\leq\mathsf{k}(G)+\frac{1}{P^{-}(\mbox{exp}(G))},
\]
so $\mathsf{k}(G)$ and $\mathsf{K}(G)$ are closely related. These
inequalities follow by observing that removing any element of an irreducible
indexed sequence leaves a zero-sum free indexed sequence. Write $G$
as a direct sum of prime power order cyclic groups 
\[
{\displaystyle G=\bigoplus_{i=1}^{r}C_{p_{i}^{\alpha_{i}}}},
\]
and define

\[
\mathsf{k}^{*}(G)=\sum_{i=1}^{r}\Big(1-\frac{1}{p_{i}^{\alpha_{i}}}\Big).
\]

This is the conjectured value of $\mathsf{k}(G)$. We choose to study
the little cross number because our methods apply more cleanly to
it; maximal cross number irreducible indexed sequences generally have
an extra term of order $\mbox{exp}(G)$, unlike maximal zero-sum free
indexed sequences or UFIS's, which generally only have terms of prime
power order. 

However, information about $\mathsf{k}(G)$ is weaker than information
about $\mathsf{K}(G)$, since if 
\[
\mathsf{K}(G)=\mathsf{K}^{*}(G)=\mathsf{k}^{*}(G)+\frac{1}{\mbox{exp}(G)},
\]
its conjectured value, then $\mathsf{k}(G)=\mathsf{k}^{*}(G)$ for
that group $G$ as well. 

Krause and Zahlten conjectured the following \cite{KZ}, towards which
the most recent progress has been the results of Geroldinger and Grynkiewicz
on the structure of maximal cross number irreducible indexed sequences
\cite{GGr}.
\begin{conjecture}
\label{conj:littlecross}The equality $\mathsf{K}(G)=\mathsf{K}^{*}(G)$
holds for all finite abelian groups $G$, and therefore $\mathsf{k}(G)=\mathsf{k}^{*}(G)$
for all $G$ as well.
\end{conjecture}
Similarly, define 
\[
\mathsf{K}_{1}^{*}(G)=\sum_{i=1}^{r}\frac{p_{i}^{\alpha_{i}}-1}{p_{i}^{\alpha_{i}}-p_{i}^{\alpha_{i}-1}}.
\]

Gao and Wang made the analogous conjecture about the $\mathsf{K}_{1}$
constant \cite{GW}.
\begin{conjecture}
\label{conj:cross}The equality $\mathsf{K}_{1}(G)=\mathsf{K}_{1}^{*}(G)$
holds for all finite abelian groups $G$.
\end{conjecture}
Just as $\mathsf{K}(G)$ is the cross number variant of $\mathsf{D}(G)$,
so too $\mathsf{K}_{1}(G)$ is the analog of the \emph{Narkiewicz
constant} $\mathsf{N}_{1}(G)$ of $G$, which is the maximal length
of a UFIS over $G$. Narkiewicz defined this latter constant to quantify
non-unique factorization in domains without unique factorization \cite{GG2,GGW,GH,N,NS}.

It is not difficult to see that indexed sequences of sufficient length
or cross number over any nontrivial $G$ will not be zero-sum free,
irreducible, or unique factorization, so all of the above group invariants
are finite.

In this paper we specifically study the constants $\mathsf{k}(G)$
and $\mathsf{K}_{1}(G)$. In both cases, the lower bound is known
by construction \cite{GW,KZ}, so it suffices to prove 
\[
\mathsf{k}(G)\leq\mathsf{k}^{*}(G),
\]
and 
\[
\mathsf{K}_{1}(G)\leq\mathsf{K}_{1}^{*}(G),
\]
respectively. These have already been shown in the following special
cases.
\begin{thm}
\label{thm:previouslittlecross}If $G$ is a group of one of the following
forms, then $\mathsf{K}(G)=\mathsf{K}^{*}(G)$, and hence $\mathsf{k}(G)=\mathsf{k}^{*}(G)$.\end{thm}
\begin{enumerate}
\item \cite{Ge2} $G$ is a finite abelian $p$-group.
\item \cite{GS1} $G=C_{p^{m}}\oplus C_{p^{n}}\oplus C_{q}^{s}$ with distinct
primes $p,q$ and $m,n,s\in\mathbb{N}$.
\item \cite{GS1} $G=\oplus_{i=1}^{r}C_{p_{i}}\oplus C_{q}^{s}$ with distinct
primes $p_{1},\ldots,p_{r},q$, and integers $n_{1},\ldots,n_{r},s\in\mathbb{N}$,
such that either $r\leq3$ and $p_{1}p_{2}\cdots p_{r}\neq30$ or
$p_{k}\geq k^{3}$ for every $1\leq k\leq r$.\end{enumerate}
\begin{thm}
\label{thm:previouscross}\cite{K} If $G$ is a group of one of the
following forms, then $\mathsf{K}_{1}(G)=\mathsf{K}_{1}^{*}(G)$.\end{thm}
\begin{enumerate}
\item $C_{p^{m}}\oplus C_{p}$, $p$ prime;
\item $C_{p^{m}}\oplus C_{q}^{2}$, $p$, $q$ distinct primes;
\item $C_{p^{m}}\oplus C_{r}^{n}$, $p\neq r$ prime, $r=2,3$.
\end{enumerate}
To improve on these theorems, we develop structural results on the
extremal indexed sequences of interest. The following definition,
motivated by an argument of Girard \cite{Gi}, will be helpful to
introduce.
\begin{defn}
A zero-sum free indexed sequence $S$ over $G$ is \emph{dense} if
$\mathsf{k}(S)=\mathsf{k}(G)$ and $|S|=\min\{|T|:\mathsf{k}(T)=\mathsf{k}(G),T\mbox{ is zero-sum free}\}$.

Similarly, UFIS $S$ is \emph{dense }if $\mathsf{k}(S)=\mathsf{K}_{1}(G)$
and $|S|=\min\{|T|:\mathsf{k}(T)=\mathsf{K}_{1}(G),T\mbox{ is a UFIS}\}.$
\end{defn}
In Section \ref{sec:densestructure}, we will show that subject to
certain conditions dense indexed sequences of both types have few
elements of each order, and mostly elements of prime power order.
In the rest of this introduction we define the conditions under which
these arguments hold and state our main results. 

For convenience, we will call a term whose order has more than one
prime factor a \emph{cross term}, following Kriz \cite{K}. The main
goal of our arguments is to show the nonexistence of cross terms in
dense indexed sequences over certain groups.

We introduce the following definition, which will prove essential
in Sections \ref{sec:densestructure} and \ref{sec:maintheorem}.
It turns out that when the primes dividing $\mbox{exp}(G)$ are far
apart, it is easier to bound the cross terms in dense indexed sequences
over $G$.
\begin{defn}
\label{def:wide}A prime $p$ is \emph{wide with respect to} a positive
integer $n$ if $p\nmid n$ and, given that the prime factorization
of $n$ is $n=q_{1}^{\alpha_{1}}\cdots q_{r}^{\alpha_{r}}$, the inequality
\begin{equation}
\frac{p}{p-1}\geq\prod_{j=1}^{r}\frac{q_{j}^{\alpha_{j}+1}-1}{q_{j}^{\alpha_{j}+1}-q_{j}^{\alpha_{j}}}\label{eq:wide}
\end{equation}
holds. In such a case we write $p\prec n$. The empty product is taken
to be $1$. A positive integer $n=q_{1}^{\alpha_{1}}q_{2}^{\alpha_{2}}\cdots q_{r}^{\alpha_{r}}$
is \emph{wide} if, assuming that $q_{1}<q_{2}<\cdots<q_{r}$, for
each $i\in[1,r-1]$, 
\[
q_{i}\prec\prod_{j=i+1}^{r}q_{j}^{\alpha_{j}}.
\]

A prime $p$ is \emph{$2$-wide with respect to} $n=q_{1}^{\alpha_{1}}\cdots q_{r}^{\alpha_{r}}$
if $p\nmid n$ and
\begin{equation}
\frac{p^{2}+2p-2}{p^{2}}\geq\prod_{j=1}^{r}\frac{q_{j}^{\alpha_{j}+1}-1}{q_{j}^{\alpha_{j}+1}-q_{j}^{\alpha_{j}}},\label{eq:2wide}
\end{equation}
in which case we write $p\prec_{2}n$. Also, $n$ is \emph{$2$-wide}
if, given $q_{1}<q_{2}<\cdots<q_{r}$, we have for each $i\in[1,r-1]$,
\[
q_{i}\prec_{2}\prod_{j=i+1}^{r}q_{j}^{\alpha_{j}}.
\]

\end{defn}
The main drawback of considering cross number instead of length is
that when $\mbox{exp}(G)$ has many prime factors cross number is
difficult to handle. Our methods extend previous results about $\mathsf{k}(G)$
and $\mathsf{K}_{1}(G)$ when $\mbox{exp}(G)$ has a small number
of prime factors to cases where $\exp(G)$ is wide and $2$-wide,
respectively.

Our main result towards the Conjecture \ref{conj:littlecross} is
the following.
\begin{thm}
\label{thm:littlecross}If $G$ is a finite abelian group and $p$
is a prime satisfying $p\prec\mbox{exp}(G)$, then 
\[
\mathsf{k}(C_{p^{\alpha}}\oplus G)=\mathsf{k}(G)+\mathsf{k}(C_{p^{\alpha}})
\]
for all $\alpha\in\mathbb{N}$. In particular, if $\mathsf{k}(G)=\mathsf{k}^{*}(G)$
then $\mathsf{k}(C_{p^{\alpha}}\oplus G)=\mathsf{k}^{*}(C_{p^{\alpha}}\oplus G)$
as well.
\end{thm}
We can prove the following by applying this result to Theorem \ref{thm:previouslittlecross}
part (1).
\begin{cor}
If $G$ is a finite abelian group, $\exp(G)$ is wide, and the $p$-components
of $G$ are all cyclic except possibly the $ $$P^{+}(\exp(G))$-component,
then $\mathsf{k}(G)=\mathsf{k}^{*}(G)$.
\end{cor}
Explicitly, these are the groups of the form $G=C_{p_{1}^{\alpha_{1}}}\oplus C_{p_{2}^{\alpha_{2}}}\oplus\cdots\oplus C_{p_{r-1}^{\alpha_{r-1}}}\oplus H_{p_{r}}$,
where $H_{p_{r}}$ is an arbitrary finite abelian $p_{r}$-group.
Similarly, if we apply Theorem \ref{thm:littlecross} to Theorem \ref{thm:previouslittlecross}
part (2), we get the following. For $a,b\in\mathbb{N}$, the interval
notation $[a,b]$ will denote the sets of integers $\{m:a\leq m\leq b\}$.
\begin{cor}
If $G$ is a finite abelian group of the form 
\[
G=C_{p_{1}^{\alpha_{1}}}\oplus C_{p_{2}^{\alpha_{2}}}\oplus\cdots\oplus C_{p_{r-2}^{\alpha_{r-2}}}\oplus C_{p_{r-1}^{\alpha_{r-1}}}\oplus C_{p_{r-1}^{\alpha_{r-1}'}}\oplus C_{p_{r}}^{s},
\]
where the $p_{i}$ are distinct primes satisfying $p_{i}\prec p_{i+1}^{\alpha_{i+1}}\cdots p_{r-1}^{\alpha_{r-1}}p_{r}$,
for all $i\in[1,r-2]$,$\alpha_{r-1}\geq\alpha'_{r-1}$, and $s$
is a nonnegative integer, then $\mathsf{k}(G)=\mathsf{k}^{*}(G)$.
\end{cor}
Our main theorem towards the $\mathsf{K}_{1}$ conjecture, proved
in Section \ref{sec:maintheorem}, is the following.
\begin{thm}
\label{thm:additive}If $G$ is a finite abelian group and $p$ is
a prime with $p\prec_{2}\mbox{exp}(G)$, then
\[
\mathsf{K}_{1}(C_{p^{\alpha}}\oplus G)=\mathsf{K}_{1}(C_{p^{\alpha}})+\mathsf{K}_{1}(G)
\]
for all $\alpha\in\mathbb{N}$. In particular, if $\mathsf{K}_{1}(G)=\mathsf{K}_{1}^{*}(G)$
then $\mathsf{K}_{1}(C_{p^{\alpha}}\oplus G)=\mathsf{K}_{1}^{*}(C_{p^{\alpha}}\oplus G)$
as well. 
\end{thm}
This theorem, combined with the currently known values of $\mathsf{K}_{1}$
listed in Proposition \ref{thm:previouscross}, proves Conjecture
\ref{conj:cross} for the following cases.
\begin{cor}
\label{cor:main} If n is a $2$-wide positive integer, and $q=P^{+}(n)$,
then $\mathsf{K}_{1}(C_{n})=\mathsf{K}_{1}^{*}(C_{n})$ and $\mathsf{K}_{1}(C_{n}\oplus C_{q})=\mathsf{K}_{1}^{*}(C_{n}\oplus C_{q})$. 
\end{cor}
These are the first cases of Conjecture \ref{conj:cross} known to
be true when $\mbox{exp}(G)$ has arbitrarily many prime factors.
In Section \ref{sec:p-groups} we will strengthen this result to the
case when $q$ is the second largest prime factor of $n$.

In Section \ref{sec:densestructure}, we prove the key structural
lemmas that are necessary to the proper bounding of dense indexed
sequences over the groups under consideration. Using this machinery,
we will prove our main theorems in Section \ref{sec:maintheorem}.
Additionally, in Section \ref{sec:p-groups} we will use a similar
argument to derive an inductive result on $p$-groups, generalizing
a theorem of Kriz \cite{K}. This will also strengthen Corollary \ref{cor:main}
to account for $q$ being the second largest prime factor of $n$.
Finally, in Section \ref{sec:elementarypgroups} we pose two conjectures
about dense UFIS's over elementary $p$-groups $C_{p}^{k}$, which
are a major roadblock to the resolution of Conjecture \ref{conj:cross},
and discuss generalizing zero-sum problems by picking other weighting
functions for indexed sequences instead of the length or cross number.

\section{The Structure of Dense Indexed Sequences\label{sec:densestructure}}

Henceforth, we will decompose a finite abelian group $G$ in the canonical
form 
\begin{equation}
G=\bigoplus_{i=1}^{r}\bigoplus_{j=1}^{k_{i}}C_{p_{i}^{\alpha_{i,j}}},\label{eq:canonicalprimeform}
\end{equation}
where $p_{1},p_{2},\ldots,p_{r}$ are distinct primes and for each
$i\in[1,r]$, we assume $\alpha_{i,1}\geq\alpha_{i,2}\geq\alpha_{i,3}\geq\cdots\geq\alpha_{i,k_{i}}$.
If $n>k_{i}$ then $\alpha_{i,n}$ is taken to be zero.

We generalize a definition from \cite{BCMP}. 
\begin{defn}
By \emph{amalgamating} a subsequence $T$ of a indexed sequence $S$
we mean the operation $S\mapsto ST^{-1}(\overline{\sigma}(T),n)$,
i.e. that of replacing $T$ with its sum. The index $n$ is the smallest
nonnegative integer for which $ST^{-1}(\overline{\sigma}(T),n)$ is
squarefree.
\end{defn}
Amalgamating any subsequence of a UFIS preserves its unique factorization,
and amalgamating any subsequence of a zero-sum free indexed sequence
keeps it zero-sum free. We will show that amalgamation can often be
performed without decreasing $\mathsf{k}(S)$. 

Our key result is the following ``Amalgamation Lemma,'' which has
two parts, one for each type of indexed sequence. Informally, we say
that an indexed sequence $S$ contains $n$ elements of order $\ell$
if
\[
\sum_{\mbox{ord}(g)=\ell}\overline{v}_{g}(S)=n.
\]

\begin{lem}
(Amalgamation Lemma.)\label{lem:amalgamation} Let $G$ be a group
of the form \eqref{eq:canonicalprimeform}, and suppose that $\alpha_{i,1}>\alpha_{i,2}$
for some $i\in[1,r]$. Let $\ell$ be a positive integer divisible
by $p_{i}^{\alpha_{i,2}+1}$. 

If $S$ is a dense zero-sum free indexed sequence over $G$, then
$S$ contains at most $p_{i}-1$ elements of order $\ell$. 

If $S$ is a dense UFIS over $G$, then $S$ contains at most $p_{i}$
elements of order $\ell$.
\end{lem}
For the proof of the UFIS case, we will require the following lemma,
which generalizes Lemma 2.2 from \cite{GGW}.
\begin{lem}
\label{lem:intersect}If $G$ is a finite abelian group and $S$ is
an indexed sequence over $G$, then $S$ divides a UFIS if and only
if for any two zero-sum subsequences $U$ and $V$ of $S$, $\gcd(U,V)$
is also zero-sum.\end{lem}
\begin{proof}
In one direction, suppose $S$ divides a UFIS $T$ and $S$ has two
zero-sum subsequences $U$ and $V$ for which $\overline{\sigma}(\gcd(U,V))\neq0$.
In particular, there are some irreducible subsequences $U'|U$ and
$V'|V$ which have nontrivial greatest common divisor. Then, factorizations
of $T(U')^{-1}$ and $T(V')^{-1}$ give rise to distinct factorizations
of $T$, so we get 
\[
|\pi^{-1}(T)|\geq|\pi^{-1}(T(U')^{-1})|+|\pi^{-1}(T(V')^{-1})|\geq2,
\]
contradicting the unique factorization of $T$.

In the other direction, suppose that $S$ has no two zero-sum subsequences
with nonzero sum greatest common divisor. Then, we claim that either
$S$ is zero-sum and thus already a UFIS, or else $S(-\overline{\sigma}(S),n)$
is a UFIS, where we choose $n$ so that $S(-\overline{\sigma}(S),n)$
squarefree.

In the first case, if $S$ is already zero-sum then it must have unique
factorization. For, if it had two distinct irreducible factorizations
$S=U_{1}U_{2}\cdots U_{m}=V_{1}V_{2}\cdots V_{n}$, then some $V_{i}$
intersects $U_{1}$ but $\gcd(V_{i},U_{1})$ cannot be zero-sum, contradicting
our assumption.

In the second case, suppose that $T=S(-\overline{\sigma}(S),n)$ does
not have unique factorization. We can reduce to the first case by
showing that no two zero-sum subsequences of $T$ have nonzero sum
greatest common divisor. Suppose otherwise; let two zero-sum subsequences
$U$ and $V$ of $T$ satisfy $\overline{\sigma}(\gcd(U,V))\neq0$.
If $(-\overline{\sigma}(S),n)$ divides $U$, we can replace $U$
by $TU^{-1}$, and assume $(-\overline{\sigma}(S),n)$ does not divide
$U$. Similarly, we assume $(-\overline{\sigma}(S),n)$ does not divide
$V$. Thus we have two zero-sum subsequences of $T(-\overline{\sigma}(S),n)^{-1}=S$
whose greatest common divisor is not zero-sum. This is a contradiction,
so we're done.
\end{proof}
We remark that Lemma \ref{lem:intersect} is still true if we stipulate
that $U$ and $V$ are irreducible, with the same proof. This is convenient
for the following argument.
\begin{proof}
(of Lemma \ref{lem:amalgamation}) We assume $\ell|\exp(G)$, or the
lemma is trivial.

Here is some notation and motivation that we will use in both cases.
Let 
\[
S_{\ell}=\prod_{\mbox{ord}(g)=\ell}g^{v_{g}(S)},
\]
where the product is over all $g\in G^{\bullet}\times\mathbb{N}$
of order $\ell$, and define $H_{k}$ to be the subgroup of $G$ consisting
of all elements with order dividing $k$, for any positive integer
$k|\exp(G)$. Now, $S_{\ell}$ can be thought of as an indexed sequence
over $H_{\ell}$. We show that if the conditions of the lemma are
false then there is some subsequence of $S_{\ell}$ of length at most
$p_{i}$ that can be amalgamated into one nonzero element of $G$
of order dividing $\ell/p_{i}$. Clearly, this amalgamation operation
does not decrease $\mathsf{k}(S)$, but decreases $|S|$, so $S$
must not be dense, which is a contradiction. 

Consider the quotient map $Q:H_{\ell}\rightarrow H_{\ell}/H_{\ell/p_{i}}$,
the codomain being isomorphic to $C_{p_{i}}$ because $p_{i}^{\alpha_{i,2}+1}|\ell$,
and extend it via the unlabelling homomorphism to a monoid homomorphism
$\overline{Q}:\mathcal{F}(H_{\ell}\times\mathbb{N})\rightarrow\mathcal{F}(C_{p_{i}}\times\mathbb{N})$.
Let $T=\overline{Q}(S_{\ell})$.

In the case that $S$ is a dense zero-sum free indexed sequence, suppose
$|T|\geq p_{i}$. Then, it is easy to show that $T$ has a nontrivial
zero-sum subsequence $T_{0}$ with length at most $p_{i}$. Let $S_{0}\in\overline{Q}^{-1}(T_{0})$
be some corresponding subsequence of $S$, with $\overline{\sigma}(S_{0})\neq0$
since $S$ is zero-sum free. 

But $Q(\overline{\sigma}(S_{0}))=0$, so it follows that $\overline{\sigma}(S_{0})\in\ker Q=H_{\ell/p_{i}}$,
and so $\mbox{ord}(\overline{\sigma}(S_{0}))\mid\frac{\ell}{p_{i}}$.
Thus $S_{0}$ can be amalgamated in the fashion described, and we
have a contradiction. This proves that any dense zero-sum free indexed
sequence $S$ must have at most $p_{i}-1$ elements of order $\ell$.

If $S$ is a dense UFIS, suppose $|T|\geq p_{i}+1$. We claim that
$T$ must have two intersecting irreducible subsequences $T_{0}$
and $T_{0}'$. Otherwise, it is a consequence of Lemma \ref{lem:intersect}
that $T$ divides a UFIS. However, $|T|>p_{i}=N_{1}(C_{p_{i}})$,
so this is impossible.

Let $S_{0}\in\overline{Q}^{-1}(T_{0})$ and $S_{0}'\in\overline{Q}^{-1}(T_{0}')$.
Because $K(C_{p_{i}})=p_{i}$, both $|S_{0}|=|T_{0}|\leq p_{i}$ and
$|S_{0}'|=|T_{0}'|\leq p_{i}$ must hold. Because 
\[
Q(\overline{\sigma}(\gcd(S_{0},S_{0}')))=\overline{\sigma}(\gcd(T_{0},T_{0}'))\neq0,
\]
we find that $\gcd(S_{0},S_{0}')$ is not zero-sum. Therefore, if
$\overline{\sigma}(S_{0})=\overline{\sigma}(S_{0}')=0$, then $S$
cannot be a UFIS, by Lemma \ref{lem:intersect}. We may assume without
loss of generality that $\overline{\sigma}(S_{0})\neq0$. Then $S_{0}$
is the subsequence of $S_{\ell}$ we are seeking to amalgamate, and
it satisfies all the desired conditions. It follows by contradiction
that $|S_{\ell}|\leq p_{i}$, as desired.
\end{proof}
As a consequence of Lemma \ref{lem:amalgamation} we have a bound
on the number of cross terms in dense indexed sequences. To eliminate
them altogether, we need a stronger hypothesis, namely that of wide
(or $2$-wide) exponent.
\begin{lem}
\label{lem:manyprimestructure} Let $G$ be of the form \eqref{eq:canonicalprimeform}
with $\alpha_{1,1}>\alpha_{1,2}$, and let $a\in[\alpha_{1,2}+1,\alpha_{1,1}]$$ $. 

If $S$ is a dense zero-sum free indexed sequence over $G$ and 
\[
p_{1}\prec p_{2}^{\alpha_{2,1}}p_{3}^{\alpha_{3,1}}\cdots p_{r}^{\alpha_{r,1}},
\]
then $S$ contains at least $p_{1}-1$ elements of order $p_{1}^{a}$. 

If $S$ is a dense UFIS over $G$ and 
\[
p_{1}\prec_{2}p_{2}^{\alpha_{2,1}}p_{3}^{\alpha_{3,1}}\cdots p_{r}^{\alpha_{r,1}},
\]
then $S$ contains at least $p_{1}-1$ elements of order $p_{1}^{a}$.\end{lem}
\begin{proof}
We start with the zero-sum free case. Suppose for the sake of contradiction
that $S$ contains at most $p_{1}-2$ elements of order $p_{1}^{a}$.
Remove the subsequence $S'$ of $S$ consisting of all its elements
with order divisible by $p_{1}^{a}$, and replace them by the indexed
sequence $T$ consisting of $p_{1}-1$ elements of each order $p_{1}^{a},p_{1}^{a+1},\ldots,p_{1}^{\alpha_{1,1}}$.
The result is a sequence $S_{1}=S(S')^{-1}T$.

If $e$ is a generator of the component $C_{p_{1}^{\alpha_{1,1}}}$
in $G$, then we define this indexed sequence $T$ as follows. First,
define the sequence
\[
T_{0}=e^{p_{1}-1}[p_{1}e]^{p_{1}-1}\cdots[p_{1}^{\alpha_{1,1}-a}e]^{p_{1}-1}.
\]
Then, choose $T\in\alpha^{-1}(T_{0})$ with $ST$ squarefree. Since
$T$ has no subsequence sums with order not divisible by $p_{1}^{a}$,
$S_{1}$ is still zero-sum free. It remains to show that $\mathsf{k}(S_{1})\geq\mathsf{k}(S)$,
or equivalently that $\mathsf{k}(T)\geq\mathsf{k}(S')$. Note that
by the Amalgamation Lemma, we get an upper bound on $\mathsf{k}(S')$
simply by bounding the number of terms of each order, from which we
can show
\[
\mathsf{k}(T)-\mathsf{k}(S')\geq\sum_{t=a}^{\alpha_{1,1}}\frac{p_{1}-1}{p_{1}^{t}}-\Big(\frac{p_{1}-2}{p_{1}^{a}}+\sum_{1<d|n_{1}}\frac{p_{1}-1}{dp_{1}^{a}}\Big),
\]
where $n_{1}=n/p_{1}^{a}$. Thus it suffices to show that if $n'=n/p_{1}^{\alpha_{1,1}}$,
then
\[
1\geq(p_{1}-1)\Big(\sum_{d|n'}\frac{1}{d}-1\Big).
\]

Now, it is easy to calculate given the factorization $n'=p_{2}^{\alpha_{2,1}}\cdots p_{r}^{\alpha_{r,1}}$
that
\[
\prod_{j=2}^{r}\frac{p_{j}^{\alpha_{j,1}+1}-1}{p_{j}^{\alpha_{j,1}+1}-p_{j}^{\alpha_{j,1}}}=\sum_{d|n'}\frac{1}{d},
\]
so it suffices to have 
\[
\frac{p_{1}}{p_{1}-1}\geq\prod_{j=2}^{r}\frac{p_{j}^{\alpha_{j,1}+1}-1}{p_{j}^{\alpha_{j,1}+1}-p_{j}^{\alpha_{j,1}}}.
\]

This is exactly the wideness condition we assumed to be true. Finally,
it is clear that $|T|<|S'|$ if $\mathsf{k}(T)=\mathsf{k}(S')$, so
since we have $\mathsf{k}(T)\geq\mathsf{k}(S')$, this contradicts
the denseness of $S$. Thus $S$ has exactly $p_{1}-1$ elements of
order $p_{1}^{a}$.

For the second case, that $S$ is a dense UFIS, suppose for the sake
of contradiction that $S$ contains at most $p_{1}-2$ elements of
order $p_{1}^{a}$. As before we replace the subsequence $S'$ of
all elements of $S$ with order divisible by $p_{1}^{a}$ by the indexed
sequence $T'$ consisting of $p_{1}$ elements of each order $p_{1}^{a},p_{1}^{a+1},\ldots,p_{1}^{\alpha_{1,1}}$.
Explicitly, this indexed sequence is chosen so that $S(S')^{-1}T'$
is squarefree and
\[
\alpha(T')=e^{p_{1}-1}[(1-p_{1})e][p_{1}e]^{p_{1}-1}[p_{1}(1-p_{1})e]\cdots[p_{1}^{\alpha_{1,1}-a}e]^{p_{1}}.
\]

After replacement the indexed sequence can still be extended to a
UFIS using Lemma \ref{lem:intersect}, if it is not one already.

Again, we need to show $\mathsf{k}(S(S')^{-1}T')\geq\mathsf{k}(S)$,
i.e. $\mathsf{k}(T')\geq\mathsf{k}(S')$. Transforming as in the zero-sum
free case, except that $T'$ now has at least $2$ more elements of
order $p_{1}^{a}$ than does $S'$, the requirement is exactly the
$2$-wideness condition we assumed. This is a contradiction once more,
so $S$ has at least $p_{1}-1$ elements of order $p_{1}^{a}$.
\end{proof}
Lemma \ref{lem:manyprimestructure} is the key result that makes the
inequalities in the definitions of wideness and $2$-wideness necessary.
Henceforth, we always assume $p_{1}$ is the smallest prime dividing
$\mbox{exp}(G)$. We will mostly consider the case $k_{1}=1$, which
allows us to pick any $a\in[1,\alpha_{1,1}]$ for Lemma \ref{lem:manyprimestructure},
but in Section \ref{sec:p-groups} we will also investigate the general
case.

\section{Proof of Main Theorem\label{sec:maintheorem}}

Combining the structural results of Section \ref{sec:densestructure},
we have enough to prove our main theorems. Let $p$ and $G$ satisfy
the conditions of Theorems \ref{thm:littlecross} and \ref{thm:additive},
so that $p\prec\mbox{exp}(G)$ or $p\prec_{2}\mbox{exp}(G)$, respectively.
Write $p_{1}=p$, $\alpha_{1,1}=\alpha$, and let the prime divisors
of $\mbox{exp}(G)$ be $p_{2},p_{3},\ldots,p_{r}$. Define $G'=C_{p_{1}^{\alpha_{1,1}}}\oplus G$.
\begin{proof}
(of Theorem \ref{thm:littlecross}) Let $S$ be a dense zero-sum free
indexed sequence over $G'$. By the Amalgamation Lemma and Lemma \ref{lem:manyprimestructure},
we find that $S$ has at exactly $p_{1}-1$ terms of each order $p_{1}^{a}$,
where $a\in[1,\alpha_{1,1}]$. Let $H=C_{p_{1}^{\alpha_{1,1}}}$ be
the $p_{1}$-component of $G'$, and let $S'=S_{H}$. Then $\mathsf{k}(S')=\mathsf{k}(H)=\mathsf{k}^{*}(H)$.
It is not difficult to show from here that $S'$ has full sumset in
$H$.

But then, we find that $S(S')^{-1}$ cannot have any subsums lying
in $H\backslash\{0\}$, or else together with some subsequence of
$S'$ one could form a zero-sum subsequence of $S$. Therefore, we
find that even after projecting $G'\rightarrow G$, the image of $S(S')^{-1}$
is still zero-sum free. As a result, $\mathsf{k}(S(S')^{-1})\leq\mathsf{k}(G)$,
since projection cannot decrease cross number. We have
\[
\mathsf{k}(S)\leq\mathsf{k}(H)+\mathsf{k}(G)=\mathsf{k}(C_{p_{1}^{\alpha_{1,1}}})+\mathsf{k}(G),
\]
as desired. It follows that $\mathsf{k}(G')=\mathsf{k}^{*}(G')$ if
$\mathsf{k}(G)=\mathsf{k}^{*}(G)$, since $\mathsf{k}^{*}$ is additive
over direct sums.
\end{proof}
The proof of Theorem \ref{thm:additive} is slightly more subtle,
because Lemma \ref{lem:manyprimestructure} is too weak by itself.
\begin{proof}
(of Theorem \ref{thm:additive}) Let $S$ be a dense UFIS over $G'$.
It follows from Lemma \ref{lem:manyprimestructure} that $S$ contains
at least $p_{1}-1$ terms of each order $p_{1}^{a}$, for $a\in[1,\alpha_{1,1}]$.
Let $H=C_{p_{1}^{\alpha_{1,1}}}$ be the $p_{1}$-component of $G$,
and let $S'=S_{H}$. We have already that $S'$ is a subsequence of
a UFIS over $H$ and contains at least $p_{1}-1$ terms of each order
dividing $p_{1}^{\alpha_{1,1}}$. It is not difficult to show that
$S'$ must have full sumset over $H$ under these conditions.

Now, we show that in fact $S'$ has exactly $p_{1}$ elements of each
of these orders. Suppose this is not the case, and choose $a$ to
be minimal for which $S'$ contains only $p_{1}-1$ terms of order
$p_{1}^{a}$, forming a subsequence $S''|S'$. We may assume $\overline{\sigma}(S'')\neq0$,
since if $\overline{\sigma}(S'')=0$ then we could replace $S''$
in $S'$ with a sequence of $p_{1}$ terms of the same order, increasing
the cross number of $S'$. 

Because $S$ has $p_{1}$ elements of each order lower than $a$,
if $H'=C_{p_{1}^{a}}$ then $S_{H'}$ has full sumset over $H'$.
Because the elements of lower order all sum to zero, there can be
at most one zero-sum free subsequence of $S(S_{H'})^{-1}$ which has
sum in $H'$ and this sum must be of order exactly $p_{1}^{a}$; otherwise,
$S$ would have two distinct intersecting irreducible subsequences.

Now, by removing one element $(b,m)$ from this zero-sum free subsequence
of $S(S_{H'})^{-1}$, if it exists, we are then allowed to insert
$(-\overline{\sigma}(S''),m)$ into $S$, to construct a indexed sequence
$S(b,m)^{-1}(-\sigma(S''),m)$ which is a subsequence of a UFIS. This
$b$ can be picked to have order divisible by $p_{1}^{a}$, since
the sum of the indexed sequence containing it must have order $p_{1}^{a}$.
But then 
\[
\mathsf{k}(S(b,m)^{-1}(-\sigma(S''),m))>\mathsf{k}(S),
\]
so it follows that our assumptions were false and $S$ contains exactly
$p_{1}$ elements of order $p_{1}^{a}$.

We can now follow the same argument as for the proof of \ref{thm:littlecross},
showing that $\mathsf{k}(S)\leq\mathsf{K}_{1}(C_{p_{1}^{\alpha_{1,1}}})+\mathsf{K}_{1}(G)$
and proving the theorem from there.
\end{proof}
As a direct consequence, Theorem \ref{thm:additive} proves Conjecture
\ref{conj:cross} for all groups of the form $C_{p^{\alpha}}\oplus C_{q^{\beta}}$
or $C_{p^{\alpha}}\oplus C_{q^{\beta}}\oplus C_{q}$, with $p<q$,
strengthening Proposition \ref{thm:previouscross} part (2). This
follows from the fact that 
\[
\frac{p^{2}+2p-2}{p^{2}}\geq\frac{q}{q-1}>\frac{q^{\beta+1}-1}{q^{\beta+1}-q^{\beta}}
\]
holds whenever $p<q$, for any $\beta\geq1$.

\section{An Inductive Result Lifting Exponents\label{sec:p-groups}}

Theorem \ref{thm:additive} is an inductive result of the form: given
a group $G$ of a certain structure, if $\mathsf{K}_{1}(G)=\mathsf{K}_{1}^{*}(G)$,
then this is also true for some group containing $G$. In this section
we derive more results of this form: if $G=\bigoplus C_{p_{i}^{\alpha_{i}}}$,
and $\mathsf{K}_{1}(G)=\mathsf{K}_{1}^{*}(G)$, we give conditions
under which we can increase the largest exponent on some prime dividing
$\mbox{exp}(G)$. Specifically, we will be able to show this when
$G$ is a $p$-group, and when $G$ is of the form $C_{p}\oplus C_{p^{\alpha}}\oplus C_{q^{\beta}}$.
For the $p$-group case, we prove the following.
\begin{prop}
\label{prop:toplift} If $p$ is prime, and $\alpha_{1}\geq\max\{\alpha_{2},\alpha_{3},\ldots,\alpha_{r}\}$,
then
\begin{equation}
\mathsf{K}_{1}(C_{p^{\alpha_{1}+1}}\oplus C_{p^{\alpha_{2}}}\oplus\cdots\oplus C_{p^{\alpha_{r}}})=\mathsf{K}_{1}(C_{p^{\alpha_{1}}}\oplus C_{p^{\alpha_{2}}}\oplus\cdots\oplus C_{p^{\alpha_{r}}})+\frac{1}{p^{\alpha_{1}}}.\label{eq:pinduct}
\end{equation}
\end{prop}
\begin{proof}
That the left hand side is at least the right hand side is immediate.
In the other direction, the proof follows by applying Lemma \ref{lem:amalgamation}
with $p_{1}=p$ and $a=\alpha_{1}+1$. Let $G$ be the group on the
left hand side of \eqref{eq:pinduct}, and $H$ be the group on the
right, identified naturally as a subgroup of $G$. We find that a
dense UFIS over $G$ can have at most $p$ elements of order $p^{\alpha_{1}+1}$,
by the Amalgamation Lemma, and the remaining terms form a subsequence
of a UFIS over $H$. The result follows.
\end{proof}
Proposition \ref{prop:toplift} generalizes Theorem 6 part (1) of
Kriz, which states that $\mathsf{K}_{1}(C_{p^{m}}\oplus C_{p}^{n})\leq\mathsf{K}_{1}(C_{p^{m}})+\mathsf{K}_{1}(C_{p}^{n+1})-1$
\cite{K}. We now extend this technique to prove the following generalization
of Theorem \eqref{thm:previouscross} part (2).
\begin{prop}
\label{prop:twoprime}If $G=C_{p}\oplus C_{p^{\alpha}}\oplus C_{q^{\beta}}$
for $p,q$ distinct primes and $\alpha,\beta$ positive integers,
then $\mathsf{K}_{1}(G)=\mathsf{K}_{1}^{*}(G)$.\end{prop}
\begin{proof}
Theorem \ref{thm:additive} is enough when $q$ is $2$-wide with
respect to $\{p\}$, i.e. when 
\begin{equation}
\frac{q^{2}+2q-2}{q^{2}}\geq\frac{p}{p-1},\label{eq:twoprimeswide}
\end{equation}
so we may assume the opposite.

Now, let $S$ be a dense UFIS over $G$, and let $S$ have $t_{a}$
elements of order $p^{a}$, for each $a\in[1,\alpha]$. Since $p$
is $2$-wide with respect to $\{q\}$, by directly applying Lemma
\ref{lem:manyprimestructure}, we find $t_{a}\geq p-1$ for each $a\geq2$.
For $a=1$ we make the following special argument.

If $t_{1}\leq2p-2$, then we can replace the subsequence of $S$ consisting
of all elements with order divisible by $p$ with the canonical example
consisting of $2p$ elements of order $p$ and $p$ of each order
$p^{a}$ with $a>1$. If the change in $\mathsf{k}(S)$ is $\delta$,
this $\delta$ is bounded below by an increase of $2/p$ plus a decrease
of at most the total cross number of the cross terms of $S$. Of these,
there are at most $q$ terms of each order $pq^{b}$ and at most $p$
of each order $p^{1+a}q^{b}$, where $a,b>0$. Thus 
\begin{equation}
\delta>\frac{2}{p}-q\sum_{b=1}^{\beta}\frac{1}{pq^{b}}-p\sum_{a,b>0}\frac{1}{p^{1+a}q^{b}}=\frac{2}{p}-\frac{q}{p(q-1)}-\frac{1}{(p-1)(q-1)}.\label{eq:deltatwoprime}
\end{equation}

It remains to prove $\delta>0$. The right side of \eqref{eq:deltatwoprime}
is always nonnegative unless $q=2$, which is impossible because we
assumed $q$ is not $2$-wide with respect to $\{p\}$. Thus, we have
shown that there are at least $2p-1$ elements of $S$ of order $p$,
and the last argument of the proof of Theorem \ref{thm:additive}
proves that $S$ must then have exactly $2p$ elements of this order.
Similarly, we find that $S$ must have exactly $p$ elements of each
order $p^{a}$ with $a>1$, so we can finish in the same way as in
the proof of Theorem \ref{thm:additive}.
\end{proof}
As a corollary, we can strengthen Corollary \ref{cor:main} by combining
this result with Theorem \ref{thm:additive}.
\begin{cor}
If $n$ is a $2$-wide positive integer, and $q$ is its largest or
second largest prime factor, then $\mathsf{K}_{1}(C_{n})=\mathsf{K}_{1}^{*}(C_{n})$
and $\mathsf{K}_{1}(C_{n}\oplus C_{q})=\mathsf{K}_{1}^{*}(C_{n}\oplus C_{q})$.
\end{cor}

\section{Concluding Remarks\label{sec:elementarypgroups}}

The study of unique factorization over elementary $p$-groups $C_{p}^{n}$
is particularly important to Conjecture \ref{conj:cross} because
every nonzero element of $C_{p}^{n}$ has order $p$, so we have $\mathsf{k}(S)=\frac{1}{p}|S|$
for all $p$. This problem is also interesting because the study of
zero-sum sequences over $p$-groups lends itself to algebraic techniques
from representation theory, see for example Chapter 5.5 of \cite{BC}
or \cite{O,P1,P2}, giving strong results that are not available in
the general case.

Study of the Narkiewicz constant $\mathsf{N}_{1}(G)$ is not limited
by the same obstacles as study of $\mathsf{K}_{1}(G)$ in the case
that $\mbox{exp}(G)$ has many prime factors; the $\mathsf{N}_{1}$
conjecture has been shown for all groups of rank at most $2$ \cite{GGW,GLP,GPZ}. 

However, in the rank direction $\mathsf{N}_{1}(G)$ and $\mathsf{K}_{1}(G)$
have proved equally difficult to compute. Moreover, it suffices to
prove the following statement to resolve the elementary $p$-group
case in both the $\mathsf{N}_{1}$ and $\mathsf{K}_{1}$ conjectures.
\begin{conjecture}
\label{conj:elementaryp}If $G=C_{p}^{n}$, then $\mathsf{N}_{1}(G)=np$.
\end{conjecture}
Conjecture \ref{conj:elementaryp} has been shown for the following
cases.
\begin{prop}
\cite{G} If $G$ is one of the following groups then $\mathsf{N}_{1}(G)=np$.\end{prop}
\begin{enumerate}
\item $G=C_{2}^{n},C_{3}^{n}$.
\item $G=C_{p}^{2}$.
\item $G=C_{5}^{3},C_{5}^{4},C_{7}^{3}$.
\end{enumerate}
Note that although case 3 was not specifically mentioned by Gao \cite{G},
the same methods in that paper apply. Now we transform Conjecture
\ref{conj:elementaryp} into a problem about finding certain subsequences
of zero-sum free indexed sequences over $G$. Again, we begin with
some structural results about maximal-length UFIS's over $C_{p}^{n}.$ 

Henceforth, let $S$ be a UFIS over $G=C_{p}^{n}$ that factors into
irreducibles as $S=U_{1}U_{2}\cdots U_{t}$. By Proposition 1 of \cite{G},
at least $ $$|S|-n(p-1)$ of the $U_{i}$ have odd length, and furthermore
\begin{equation}
\prod_{i=1}^{t}|U_{i}|\leq|G|=p^{n}.\label{eq:6}
\end{equation}

We say a $U_{i}$ is \emph{optimal} in $S$ if it cannot be replaced
by any longer irreducible indexed sequence while preserving the unique
factorization of $S$. Notice that the only property necessary to
satisfy is the following.
\begin{prop}
\cite{G} If $U_{1}$ and $S'$ are an irreducible indexed sequence
and a UFIS, respectively, then $U_{1}S'$ is a UFIS if and only if
$\Sigma(U_{1})\cap\Sigma(S')=\{0\}.$
\end{prop}
Thus $U_{i}$ is optimal in $S$ if and only if it is the longest
irreducible indexed sequence over $G$ which has sumset equal to $\{0\}\cup(G\backslash\Sigma(S'))$,
where $S'=SU_{i}{}^{-1}$.

Note that if $|U_{i}|\leq p$ for all $i$, then $|S|\leq np$ by
the concavity of the logarithm and \eqref{eq:6}. Henceforth a zero-sum
indexed sequence is called \emph{short} if it has length at most $p$
and \emph{long} otherwise. A zero-sum free indexed sequence is called
\emph{short} if it has length less than $p$ and \emph{long }otherwise.
The following conjecture would resolve the case of longer $U_{i}.$
\begin{conjecture}
\label{shortconjecture} If $T$ is a zero-sum free indexed sequence
over $G=C_{p}^{k}$, then there exists a subsequence $T_{0}$ of $T$
with length less than $p$ such that no other subsequence of $T$
has the same sum as $T_{0}$.
\end{conjecture}
Note that the statement is trivial for $|T|<p$, since we can just
take $T_{0}=T$, and the sum is unique because $T$ is zero-sum free.
Also, it suffices to show that $T$ has a proper subsequence $T_{0}$
(not necessarily with length less than $p$) satisfying this property,
since we can apply the result inductively. To see why this conjecture
implies Conjecture \ref{conj:elementaryp}, we claim that any maximal
UFIS over $G$ must have only short irreducible subsequences. Otherwise,
remove one element from a long $U_{i}$ and apply Conjecture \ref{shortconjecture}
to the remaining zero-sum free subsequence to find that $U_{i}$ can
be replaced by a pair of short irreducible indexed sequences of longer
total length.

Next, we offer the following strengthening of Conjecture \ref{shortconjecture}.
\begin{conjecture}
\label{Subgroupconjecture} If $T$ is a zero-sum free indexed sequence
over $G=C_{p}^{k}$, then there exists a cyclic subgroup $H$ of $G$
such that $T_{H}\neq1$ and $H\cap\Sigma(T(T_{H})^{-1})=\emptyset$.
\end{conjecture}
To see that Conjecture \ref{Subgroupconjecture} implies Conjecture
\ref{shortconjecture}, take $T_{0}=T_{H}$.

For the study of the group invariants $\mathsf{k}$, $\mathsf{K}$,
and $\mathsf{K}_{1}$ when $\exp(G)$ has many prime factors, our
methods rely on conditions such as wideness and $2$-wideness. We
are hopeful that a fully general solution removing these conditions,
or improvements on the inequalities in the wide and $2$-wide conditions,
can be made.

If we weight indexed sequences more severely in the following manner
we can drop the wideness and $2$-wideness conditions. Define weighting
function $f$ to be the totally multiplicative function on the positive
integers with $f(p_{i})=1/2^{i}$, where $p_{i}$ is the $i$-th prime.
If $S$ is an indexed sequence over a finite abelian group $G$ define
\[
\mathsf{k}(S,f)=\sum_{g\in G}\overline{v}_{g}(S)\cdot f(\mbox{ord}(g)).
\]

Let $\mathsf{k}(G,f),\mathsf{k}^{*}(G,f),\mathsf{K}_{1}(G,f)$ and
$\mathsf{K}_{1}^{*}(G,f)$ be defined in the natural ways. $ $It
is not difficult to show the following using our methods, in place
of Theorems \ref{thm:littlecross} and \ref{thm:additive}. Recall
that $P^{-}(n)$ is the smallest prime dividing $n$.
\begin{prop}
\label{prop:generalweight}If $p<P^{-}(\mbox{exp}(G))$ and $f$ is
defined as above, then $\mathsf{k}(C_{p^{\alpha}}\oplus G,f)=\mathsf{k}(G,f)+\mathsf{k}(C_{p^{\alpha}},f)$
and $\mathsf{K}_{1}(C_{p^{\alpha}}\oplus G,f)=\mathsf{K}_{1}(C_{p^{\alpha}},f)+\mathsf{K}_{1}(G,f)$.
\end{prop}
Note that $\mathsf{D}(G)$ and $\mathsf{N}_{1}(G)$ are, respectively,
the cases of $\mathsf{K}(G,f)$ and $\mathsf{K}_{1}(G,f)$ when $f=1$,
and no additive result like Proposition \ref{prop:generalweight}
is true for either of them. A study of other choices of weighting
functions $f$, and the structures of the dense indexed sequences
that result, might shed light on all of the conjectures we are interested
in.

\section*{Acknowledgments}

This research was conducted as part of the University of Minnesota
Duluth REU program, supported by NSF/DMS grant 1062709 and NSA grant
H98230-11-1-0224. I would like to thank Joe Gallian for his encouragement
and general advice and for making this experience possible, and Daniel
Kriz for his extensive expertise and support, and for reading over
my manuscript. Also, my gratitude goes out to program advisors Krishanu
Sankar and Sam Elder for their helpful input on the early versions
of this paper, and to the anonymous referee for meticulous comments
on its later versions.

\end{document}